\documentclass[11pt]{amsart}
\usepackage[cm]{fullpage}
\usepackage{pgfplots,stackrel, aligned-overset}

\usepackage{amssymb,amsmath,amsthm,amscd,mathrsfs,graphicx}
\usepackage[cmtip,all]{xy}
\usepackage{pb-diagram}
\usepackage{tikz}
\usetikzlibrary{arrows, arrows.meta}

\numberwithin{equation}{section}

\newtheorem{theorem}{Theorem}[section]

\newtheorem{lemma}{Lemma}[section]

\newcommand{\ve}{\varepsilon}

\theoremstyle{definition}

\theoremstyle{remark}

\begin{document}

\title{Composite media, almost touching disks \\ and the maximum principle}

\author[Y. Li]{Yanyan Li}
\address{Department of Mathematics, Rutgers University, 110 Frelinghuysen
Road, Piscataway, NJ 08854, USA}

\author[B. Weinkove]{Ben Weinkove}
\address{Department of Mathematics, Northwestern University, 2033 Sheridan Road, Evanston, IL 60208, USA.}

\thanks{Y.Y. Li is partially supported by NSF grant DMS-2247410.  B. Weinkove is partially supported by NSF grant DMS-2348846 and  the Simons Foundation.}

\begin{abstract}
We consider the setting of two disks in a domain in $\mathbb{R}^2$ which are almost touching and have finite and positive conductivities, giving rise to a divergence form elliptic equation with discontinuous coefficients.  We use the maximum principle to give a new proof of a gradient bound of Li-Vogelius.
\end{abstract}

\maketitle

\maketitle

\section{Introduction} \label{sectionintro}

Let $B^+$ and $B^-$ be two disks of radius $1$, compactly contained in a bounded open set $\Omega \subset \mathbb{R}^2$ with smooth boundary $\partial \Omega$.  Assume that $B^+$ is centered at the origin and $B^-$ is centered at $(0,-2-2\delta)$ so that the disks are a distance $2\delta$ apart.  The horizontal line $x_2=-1-\delta$ is midway between $B^-$ and $B^+$.
We allow $\delta$ to vary in $(0,\delta_0)$ for some $\delta_0>0$ and our goal is to prove results which are independent of $\delta$.

Define the conductivity,
$$a(x): =  \begin{cases} \kappa^+, \quad & x\in B^+ \\ \kappa^-, \quad & x\in B^- \\ 1, \quad & x \in \Omega \setminus( B^+ \cup B^-),  \end{cases}$$
where $\kappa^+, \kappa^-$ are fixed positive constants.  See Figure \ref{figure1}.
We  fix a smooth function $\varphi$ on $\partial{\Omega}$ and consider the Dirichlet problem 
\begin{equation} \label{equation}
\begin{split}
 (au_i)_i= {} & 0, \quad \emph{on } \Omega \\
u ={}& \varphi, \quad \emph{on } \partial \Omega,
\end{split}
\end{equation}
where we use the usual summation convention for repeated indices $i=1,2$.

\begin{figure}[h]  
\begin{tikzpicture}
    %\begin{axis}[thick,
    %    xmin=-4.5,xmax=4.5,
     %   ymin=-3.5,ymax=3.5,
     %  axis x line=middle,
     %  axis y line=middle,
    %    axis line style=<->,
      %  xlabel={$\scriptstyle x_2=-1-\delta$},
        %x tick label style={major tick length=0pt},
        %ylabel={$\scriptstyle x_2$},
       % yticklabels={,,},
      %  x tick label style={major tick length=0pt},      
     %  xticklabels={,,}
    % ]
       \draw (3.415cm,3.8cm) circle (.8cm);
 \draw (3.415cm, 1.9cm) circle (.8cm);
 \draw [->] plot [smooth cycle, tension=0.9] coordinates { (1cm,.5cm) (4.5cm, .6cm) (5.5cm,2.5cm) (6cm, 4.5cm) (2.5cm,5.2cm) (1.2cm,3.5cm)};
      %\end{axis}
 %   \draw [ultra thick] (3.417, 3.0) circle (.03cm);
    \node(A) at (3.75, 2.90cm) {$\scriptstyle 2\delta$};
        \node(C) at (2, 2.90cm) {$\scriptstyle a=1$};
       % \draw [ultra thick] (3.417, 2.7) circle (.03cm);
            \node(B) at (3.47, 3.82cm) {$\scriptstyle \, \, a=\kappa^+$};
     %   \draw [ultra thick] (3.417, 3.8) circle (.03cm);
    %\node(C) at (3.58, 2.54cm) %{$\scriptstyle\textrm{-}\varepsilon$};
    % \draw [ultra thick] (3.417, 1.9) circle (.03cm);
    \node(C) at (3.47, 1.93cm) {$\scriptstyle \, \, a=\kappa^-$};
\node(D) at (2.4, 4.2cm) {$\scriptstyle B^+$};
\node(E) at (2.35, 1.7cm) {$\scriptstyle B^-$};
\node(F) at (6cm, 4.8cm) {$\scriptstyle \Omega$};
\draw[<->] (3.417,3.0) -- (3.417,2.7);
\end{tikzpicture}
\caption{The domain $\Omega$ contains the disks $B^+$ and $B^-$, a distance $2\delta$ apart.} \label{figure1}
\end{figure}
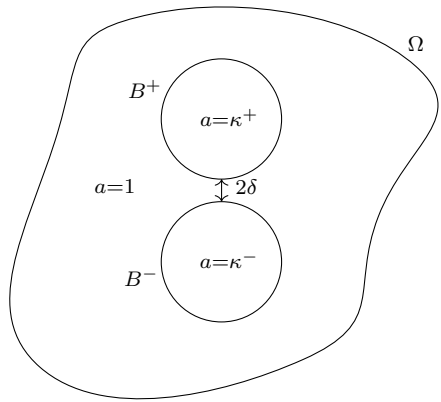

Physically, this situation arises in the study of composite media with almost touching inclusions of different conductivities.  The function $u$ represents the electric potential. From the point of view of material science it is important to estimate the magnitude of the electric field $|\nabla u|$.  Early work of Babuska-Andersson-Smith-Levin \cite{BASL} and Bonnetier-Vogelius \cite{BV} suggested that $|\nabla u|$ should remain bounded independently of $\delta$ and this was established in all dimensions and for general shaped inclusions by the first-named author and Vogelius \cite{LV}.  This was extended to systems by Li-Nirenberg \cite{LN}.  

There is also interest in the degenerate cases when $\kappa^+$ and  $\kappa^-$ tend to zero, ``the insulated conductivity problem'', or to infinity, ``the perfect conductivity problem''.  In both cases the gradient $|\nabla u|$ becomes unbounded as $\delta$ tends to zero.  These types of problems have led to a large and growing body of research, and we refer the reader to the references \cite{ACKLY, AKLLL, AKL, BLY0, BLY, BC, DL, DLY, DLY2,  KLY, K, LY, LYu, M, Y1, Y2, Y3}, a list which is far from complete.

Given the elementary nature of these problems, it is natural to ask whether there are elementary proofs. In \cite{W}, the second-named author introduced a maximum principle method to give a new proof of an estimate of Bao-Li-Yin \cite{BLY} for the insulated conductivity problem, and some new  effective bounds in dimensions $n \ge 4$ (this came after work of Li-Yang \cite{LY} in dimensions $n\ge 3$ using other techniques).  Since then, Dong-Li-Yang \cite{DLY, DLY2} have proved optimal bounds for the insulated conductivity problem by other methods.  

We  note that the maximum principle also played a role in some earlier work such as that of Bonnetier-Vogelius \cite{BV}.
 
We speculate that the maximum principle alone can be used to reprove \emph{all} the existing results in this literature, including the optimal ones.  We expect that doing so would lead to important new insights and results.

  This paper makes a step in this direction by using a maximum principle method to prove a gradient estimate for the problem (\ref{equation}) above with fixed positive constants $\kappa^+, \kappa^-$.

For $\mu>0$, we write $\Omega^{\mu} = \{ x\in \Omega \ | \ d(x, \partial \Omega) \ge \mu\}$.  Fix $\mu \in (0,1)$ sufficiently small so that $B^+\cup B^-$ is contained in $\Omega^{2\mu}$.
We give a maximum principle proof of the following result, originally due to Li-Vogelius \cite{LV}.

\begin{theorem} \label{sa} Let $u\in H^1(\Omega)$ be the unique weak solution of
\begin{equation}
\begin{split}
 (au_i)_i= {} & 0, \quad \emph{on } \Omega \\
u ={}& \varphi, \quad \emph{on } \partial \Omega.
\end{split}
\end{equation}
Then there exists a constant $C$ depending only on $\Omega$, $\mu$, $\kappa^+$ and $\kappa^-$ such that
\begin{equation} \label{nue}
 \| \nabla u \|_{L^{\infty}(\Omega^{\mu})} \le C \| \varphi \|_{L^{\infty}(\partial \Omega)}.
\end{equation}
\end{theorem}

Note that, given this, one can extend the $L^{\infty}$ bound (\ref{nue}) to the whole of $\Omega$  using 
 standard boundary estimates for harmonic functions, which can also be proved using only the maximum principle.  This gives
 $\| \nabla u \|_{L^{\infty}(\Omega)} \le C'$, for $C'$ a constant depending only on $\Omega, \mu, \kappa^+, \kappa^-$ and $\varphi$.

The outline of the paper is as follows.  In Section \ref{section2}, we briefly describe some preliminary results that we need for what follows, including the method of smooth approximation to replace the discontinuous function $a$ with a family of smooth functions $a_{\ve}$.  In Section \ref{section3}, we prove some identities for radial and tangential derivatives of $u$, which hold in a neighborhood of one of the disks.  In Section \ref{section4}, we give the proof of Theorem \ref{sa}.  Finally, in Section \ref{section5}, we end with some remarks on the difficulty of extending to dimensions $n >2$.

\section{Preliminaries} \label{section2}

\subsection{Smooth approximation} We  smoothly approximate the discontinuous function $a$ by smooth functions $a_{\ve}$ for a parameter $\ve$ with $0<\ve<<\delta$.  We assume that in an $\varepsilon$-neighborhood of $\partial B^+$,  the function $a_{\ve}=a_{\ve}(r)$ is radial, where $r$ denotes the distance from the center of $B^+$.  Moreover, we assume that $a_{\ve}$ is nonincreasing if $\kappa^+ >1$ and nondecreasing if $\kappa^+ <1$.  We do the same in an $\varepsilon$-neighborhood of $B^-$.  We assume that $a_{\ve}=a$ (hence is constant) outside of these $\varepsilon$-neighborhoods.  As $\ve \rightarrow 0$, the functions $a_{\ve}$ are  bounded and  converge pointwise almost everywhere to $a$.

Now let $u_{\ve}$ be the smooth solution of
\begin{equation} \label{euve}
\begin{split}
(a_{\ve} (u_{\ve})_i)_i= {} & 0, \quad \emph{on } \Omega \\
u_{\ve} ={}& \varphi, \quad \emph{on } \partial \Omega.
\end{split}
\end{equation}

By scaling, we may and do assume that $\| \varphi \|_{L^{\infty}(\partial \Omega)} =1$.
Then we aim to prove, using the maximum principle,
$$\sup_{\Omega^{\mu}} | \nabla u_{\ve} | \le C,$$
from which the main theorem will follow by approximation.

\medskip
\noindent
\emph{Important note:} For simplicity of notation, from now on we will drop the subscript $\varepsilon$ and write $u$ for $u_{\ve}$ and $a$ for $a_{\ve}$.

\subsection{$L^{\infty}$ bound on $u$} The maximum principle applied to (\ref{euve}) immediately gives
\begin{equation} \label{lib}
\sup_{\Omega} |u| \le \| \varphi\|_{L^{\infty}(\partial \Omega)}=1.
\end{equation}

\subsection{Classical interior estimate for harmonic functions}  Let $B_{\sigma}(x)$ be a disk of radius $\sigma$ centered at $x \in \Omega$.  If we assume that $B_{\sigma}(x)$ is  compactly contained in one of the three domains $B^+$,  $B^-$ or $\Omega \setminus (B^+\cup B^-)$ then the classical interior estimate for harmonic functions  implies that (shrinking $\ve$ if necesssary),
\begin{equation} \label{classical}
\begin{split}
\sup_{B_{\sigma/2}(x)} |\nabla u| \le {} &  \frac{C_1}{\sigma} \sup_{B_{\sigma}(x)} |u| \\
%\sup_{B_{\sigma/2}(0)} |\nabla \nabla v| \le {} &  \frac{C_2}{\sigma^2} \sup_{B_{\sigma}(0)} |v|,
\end{split}
\end{equation}
where $C_1$ is a universal constant.

\section{Tangential and radial derivatives} \label{section3}

In this section, we work in the neighborhood of $B^+$ given by those points $x$ in $\Omega$ with $x_2>-1-\frac{3}{2}\delta$. 
Analogous estimates hold in a neighborhood of $B^-$.   We define two quantities, corresponding to  tangential and radial derivatives respectively.  Define
 $$T = -x_2 u_1 + x_1 u_2$$
 and
 $$R= a (x_1u_1+x_2u_2).$$
 The next lemma shows that  $T$ solves the same equation as $u$ itself, while $R$ solves the same equation with $a$ replaced by $a^{-1}$.
 
\begin{lemma} \label{lemmaTjR} Assume $x \in \Omega$ with $x_2>-1-\frac{3}{2}\delta$.  We have, 
\begin{equation}
aT_{ii}  + a_i T_i =0 \quad \textrm{and} \quad 
a R_{ii} - a_i R_i = 0.
\end{equation}
\end{lemma}
\begin{proof}
From the equation $(au_i)_i=0$ we have, for $k=1, 2$,
\begin{equation} \label{ed}
\begin{split}
au_{ii} + a_i u_i= {} & 0 \\
au_{iik} + a_k u_{ii} + a_{ik} u_i + a_i u_{ik}= {} & 0.
\end{split}
\end{equation}
Since $a$ is radial, we have,
\begin{equation} \label{radial}
\begin{split}
-x_2 a_1 + x_1 a_2 = {} & 0 \\
-x_2 a_{11} + a_2 + x_1 a_{21} = {} & 0 \\
-a_1-x_2 a_{12} + x_1 a_{22} = {} & 0.
\end{split}
\end{equation}

Using this, we obtain 
\[
\begin{split}
aT_{ii}   & =   a( -x_2u_{1ii} + x_1 u_{2ii}) \\
 \overset{\eqref{ed}}&{=}    x_2 (a_1 u_{ii} + a_{i1}u_i + a_i u_{i1}) - x_1 (a_2 u_{ii} + a_{i2}u_i + a_i u_{i2}) \\
\overset{\eqref{radial}}&{=}  a_2u_1-a_1u_2 + a_i(x_2 u_{1i} - x_1 u_{2i}) \\
& =  -   a_i T_i,
\end{split}
\]
as required.

For $R$, compute 
\[
\begin{split}
aR_{ii} & =  a (ax_ku_k)_{ii} \\
& =  a\left( a_{ii} x_ku_k + ax_k u_{kii} + 2a_ku_k+ 2a_ix_k u_{ki} + 2au_{kk} \right) \\
\overset{\eqref{ed}}&{=} a\left( a_{ii} x_k u_k - x_k a_k u_{ii} - x_k a_{ik} u_i - x_k a_i u_{ik} + 2a_i x_k u_{ki}\right)  \\
& = a( a_{11} x_2u_2 + a_{22}x_1 u_1 - x_2 a_{12} u_1 - x_1 a_{21} u_2 + x_k a_k \frac{a_iu_i}{a} + a_i x_k u_{ki }) \\
\overset{\eqref{radial}}&{=} a ( a_1 u_1 + a_2 u_2  + a_i x_k u_{ki}) + a_k^2 x_i u_i \\
& = a_i R_i,
\end{split}
\]
where we used the first line of (\ref{radial}) to obtain $x_k a_k a_iu_i = a_k^2 x_i u_i$.
\end{proof}

Since $T$ and $R$ do not satisfy the same equation, their linear combinations $\alpha T + \beta R$ will satisfy equations which depend on $\alpha$ and $\beta$.  First, we have the following lemma.

\begin{lemma} \label{lemmaQS2}
We have
\begin{equation} \label{SQ}
\begin{split}
a_1 R_1 + a_2 R_2 = {} & a (a_2 T_1-a_1 T_2) \\
a(a_1 T_1 + a_2 T_2) = {} & -a_2 R_1+ a_1 R_2 
\end{split}
\end{equation}
\end{lemma}
\begin{proof}
We use only the equations $a_1x_2 = a_2 x_1$ and $a_i u_i = - a u_{ii}$.  To prove the first equation of (\ref{SQ}), compute
\[
\begin{split}
a_1 R_1+a_2 R_2 = {} & (a_1^2 + a_2^2) (x_1u_1+x_2u_2) + aa_1 (u_1 + x_1 u_{11} + x_2 u_{12}) \\{} &  + a a_2 (x_1 u_{12} + u_2 + x_2 u_{22}) \\
= {} & a_1^2 x_1 u_1 + a_2 a_1 x_2 u_1 + a_1 a_2 x_1 u_2 + a_2^2 x_2 u_2 + aa_i u_i \\ {} & + aa_1 x_1 u_{11} + aa_1 x_2 u_{12} + aa_2 x_1 u_{12} + aa_2 x_2 u_{22}  \\
= {} & - x_1 a_1 a u_{ii} - x_2 a_2 a u_{ii} + aa_i u_i  \\ {} & + aa_1 x_1 u_{11} + aa_1 x_2 u_{12} + aa_2 x_1 u_{12} + aa_2 x_2 u_{22} \\
= {} & - a x_1 a_1 u_{22} - a x_2 a_2 u_{11} + aa_i u_i + aa_1 x_2 u_{12} + aa_2 x_1 u_{12},
\end{split}
\]
and
\[
\begin{split}
a(a_2 T_1 - a_1 T_2) = {} & a (a_2 (-x_2 u_{11} + u_2 + x_1 u_{21}) - a_1( -u_1 - x_2 u_{12} + x_1 u_{22} )) \\
= {} & a a_i u_i - a a_2 x_2 u_{11} + a a_2 x_1 u_{21} + aa_1 x_2 u_{12} - a a_1 x_1 u_{22},
\end{split}
\]
as required.

For the second equation, compute
\[
\begin{split}
-a_2 R_1 + a_1 R_2 = {} & -a_2 a_1 (x_1 u_1+x_2 u_2) + a_1 a_2 (x_1u_1+x_2u_2) \\ {} & - aa_2 (u_1+ x_1 u_{11} + x_2 u_{21}) + aa_1 (x_1u_{12} + u_2 + x_2 u_{22})\\
= {} & a \bigg( - a_2 u_1 - a_1 x_2 u_{11} - a_2 x_2 u_{12} + a_1 x_1 u_{12} + a_1 u_2+ a_2 x_1 u_{22} \bigg) \\
= {} & a\bigg( a_1 (-x_2 u_{11} + u_2 + x_1 u_{12}) + a_2 (- u_1 - x_2 u_{12} + x_1 u_{22}) \bigg) \\
= {} & a (a_1 T_1+ a_2 T_2),
\end{split}
\]
completing the proof of the lemma.
\end{proof}

%\begin{remark} A shorter way to prove the above lemma is to use rotational symmetry and compute at a point where $x_1=0=a_1$.
%\end{remark}

We can now find the equation satisfied by a linear combination of $T$ and $R$.

\begin{lemma} \label{combination}
For $\alpha, \beta \in \mathbb{R}$, we have
\begin{equation}
(\alpha T+ \beta R)_{ii} = \left( \frac{\beta^2 a^2 - \alpha^2}{\beta^2 a^2+\alpha^2} \right) \frac{a_i}{a} (\alpha T+\beta R)_i + \frac{2\alpha \beta}{\beta^2a^2+\alpha^2} (a_2 (\alpha T+\beta R)_1 - a_1 (\alpha T+\beta R)_2).
\end{equation}
\end{lemma}
\begin{proof}
From Lemma \ref{lemmaQS2}, we have
\[
\begin{split}
\lefteqn{\left( \frac{\beta^2 a^2 - \alpha^2}{\beta^2 a^2+\alpha^2} \right) \frac{a_i}{a} (\alpha T+\beta R)_i + \frac{2\alpha \beta}{\beta^2a^2 + \alpha^2} (a_2 (\alpha T+\beta R)_1 - a_1 (\alpha T+\beta R)_2)} \\= {} & 
\left( \frac{\beta ^2 a^2 - \alpha^2}{\beta^2 a^2+\alpha^2} \right) \frac{a_i}{a} (\alpha T+\beta R)_i+ \frac{2\alpha \beta}{\beta^2a^2 + \alpha^2} \left( \frac{\alpha}{a} a_i R_i - \beta a a_i T_i  \right) \\
= {} & \left( \frac{\beta^2 a^2 - \alpha^2 - 2\beta^2 a^2}{\beta^2 a^2+\alpha^2} \right) \frac{a_i}{a} \alpha T_i + \left( \frac{\beta^2 a^2 - \alpha^2 + 2\alpha^2}{\beta^2 a^2+\alpha^2} \right) \frac{a_i}{a} \beta R_i \\
= {} & -\frac{\alpha a_i}{a} T_i + \frac{\beta a_i}{a} R_i = (\alpha T+ \beta R)_{ii},
\end{split}
\]
as required.
\end{proof}

\section{Proof of Theorem \ref{sa}} \label{section4}

%Assume that we have gradient bounds away from the narrow region between $B^+$ and $B^-$ (this part should not be difficult).  We will prove gradient bounds in the narrow region.

Define closed sets $W^+$ and $W^-$ by
\[
\begin{split}
W^+:= {} & \{ x \in \Omega \ | \ \textrm{dist}(x, \partial B^+) \le \mu, \ x_2 \ge -1-\delta\} \\
W^-:={} & \{ x \in \Omega \ | \ \textrm{dist}(x, \partial B^-) \le \mu, \ x_2\le -1-\delta\},
\end{split}
\]
which are contained in $\Omega^{\mu}$.
By the classical interior estimate (\ref{classical}), $|\nabla u|$ is uniformly bounded  on $\Omega^{\mu} \setminus  \textrm{int} (W^+ \cup W^-)$.  Hence $|\nabla u|$ is bounded on $\partial W^+$ and $\partial W^-$ except for  on the edge where $x_2=-1-\delta$ and $x_1$ is small.  

Let $c_1>0$ be the largest constant  so that the line segment 
$$S: = \{ (x_1, x_2) \ | \ -c_1\le x_1 \le c_1, \ x_2=-1-\delta \},$$
is contained in $W^+\cup W^-$.  

\begin{lemma} \label{lemma4} We have
\begin{equation} \label{firstclaim} 
\sup_{\Omega^{\mu}} |\nabla u| \le C (\sup_S|\nabla u| +  1 ).
\end{equation}
\end{lemma}

\begin{proof} The maximum principle and Lemma \ref{combination} implies that
$$\sup_{W^+} |\alpha T + \beta R|\le \sup_S | \alpha T + \beta R| + C( | \alpha| + |\beta|) \le C'(|\alpha| + |\beta|) (\sup_S |\nabla u| + 1)$$
for any $\alpha, \beta \in \mathbb{R}$. Since at any point in $W^+$, $|\nabla u|$ can be expressed as a linear combination of $\alpha T + \beta R$ for uniformly bounded $\alpha, \beta$,  it follows that
$$\sup_{W^+} |\nabla u| \le C( \sup_S |\nabla u| + 1),$$
and we obtain a similar bound for $W^-$. The result follows. \end{proof}

We now define a thin rectangular region $E$, with $S \subset E \subset W^+ \cup W^-$, by
$$E = \left\{ (x_1, x_2) \ | \ -c_1 \le x_1 \le c_1,  \ |x_2 +1+\delta | \le \frac{\delta}{2} \right\}.$$
See Figure \ref{fig2}. We will prove that $|\nabla u|$ is bounded on $E$, which by Lemma \ref{lemma4} will suffice to prove Theorem \ref{sa}.

\begin{figure}[h]
\begin{tikzpicture}
%\draw (3.415cm,4cm) circle (1.2cm);
%\draw [dashed] (6cm,8cm) circle (1.6cm);
%\draw [dashed] (6cm,3.4cm) circle (1.6cm);

  %  \node(C) at (3.58, 2.54cm) {$\scriptstyle\textrm{-}\varepsilon$};
%\node(D) at (2.1, 4.2cm) {$\scriptstyle B^+$};
%\node(E) at (2.1, 1.7cm) {$\scriptstyle B^-$};
%\node(F) at (4.43, 2.7) {$\scriptstyle c$};
%\node(G) at (2.25, 2.7) {$\scriptstyle - c$};
%\node(H) at (4.37, 3.25) {$\scriptstyle V$};
\def\firstcircle{(6cm,8cm) circle (1.6cm)}
\def\secondcircle{(6cm,8cm) circle (2.4cm)}
\def\thirdcircle{(6cm, 3.4cm) circle (2.4cm)}
\def\fourthcircle{(6cm,3.4cm) circle (1.6cm)}
\def\fifthcircle{(6cm,3.4cm) circle (2.4cm)}
\begin{scope}
\clip \firstcircle (0,0cm) rectangle (12cm, 12cm);
\clip \thirdcircle (0,0cm) rectangle (12cm, 12cm);
\draw[fill={gray!15}, draw=none]  \secondcircle;
\end{scope}
\begin{scope}
\clip \secondcircle (0,0cm) rectangle (12cm, 12cm);
\clip \fourthcircle (0,0cm) rectangle (12cm, 12cm);
\draw[fill={gray!15}, draw=none]  \fifthcircle;
\end{scope}
       \draw (6cm,8cm) circle (2cm);
 \draw (6cm, 3.4cm) circle (2cm);
 \draw [dotted] (2,5.7) -- (10, 5.7); 
 \node(H) at (10.8, 5.72) {$\scriptstyle x_2=-1-\delta$};
  \node(E) at (6.9, 5.85) {$\scriptstyle E$};
 \draw[draw=none, fill={gray!60}] (5.3, 5.55) rectangle ++ (1.4, 0.3);
 \draw[thick] (5.3, 5.7) -- (6.7, 5.7);
     \node(A) at (3.5, 9cm) {$\scriptstyle W^+$};
        \node(B) at (3.5, 2.4cm) {$\scriptstyle W^-$};
        \draw[<->] (8,8.0) -- (8.4,8);
          \node(C) at (8.19, 7.72cm) {$\scriptstyle \mu$};
\end{tikzpicture}
\caption{The light shaded regions are the closed sets $W^+$ and $W^-$.  The dark shaded region is $E$.  The horizontal line segment shown running through the center of $E$ is the set $S$.}  \label{fig2}
\end{figure}
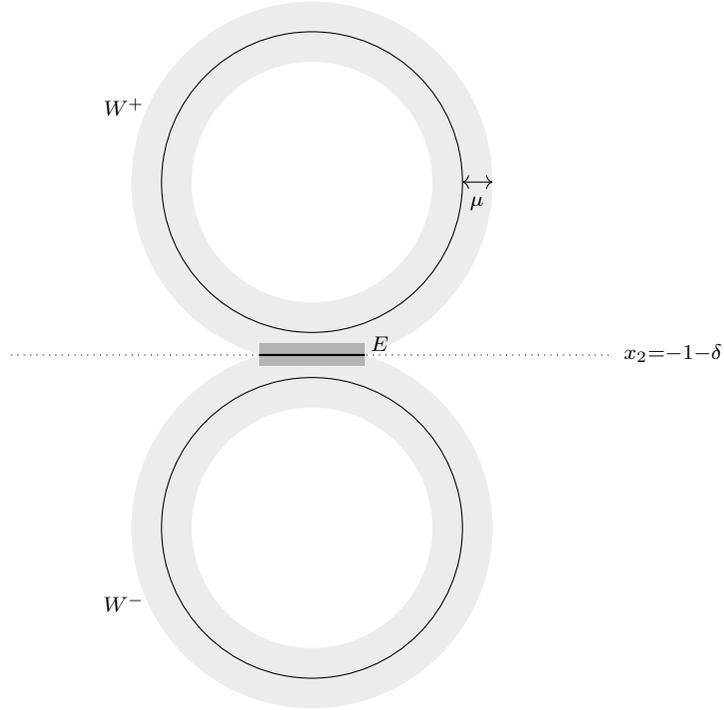

Consider the quantity $$M:= |\nabla u|^2+Au^2$$ on $E$ for a uniform constant $A>1$ to be determined later.  We will prove that $M$ is bounded on $E$.
Since $M$ is subharmonic on $E$, it achieves a maximum at a point $p$ on the boundary of $E$.  Without loss of generality, we may assume that  $p$ lies on the upper boundary of $E$ (if $p$ lies on one of the side edges we are done, since $|\nabla u|$ is bounded there).

We may assume that at $p$, there are constants $\alpha, \beta$ with $\alpha^2+\beta^2=1$ such that
\begin{equation} \label{atp}
|\nabla u(p)|^2 = \frac{(\alpha T + \beta R)^2(p)}{|p|^2}.
\end{equation}
Indeed, for fixed $x\neq 0$,  the vectors $(-x_2, x_1)/|x|$ and $(x_1, x_2)/|x|$  form an orthonormal basis for $\mathbb{R}^2$ and on $E$ we have $T=(-x_2, x_1)\cdot \nabla u$ and $R = (x_1, x_2) \cdot \nabla u$, recalling that $a=1$ there.  We now fix this $\alpha$ and $\beta$.

Assume first that $\kappa^+ \ge 1$, so that in particular $a \ge 1$ on $W^+$.  Later we will explain how to deal with the case when $0<\kappa^+ <1$.

Let  $b$ be  the smooth radial function on $W^+$ such that, writing $b=b(|x|)$,  
\begin{equation} \label{bdefn}
b(1+\ve)=0, \quad b'(|x|)= 1-a(|x|)^K\le 0,
\end{equation}
for $K$ a large uniform positive constant to be determined.
Here we are writing $a=a(|x|)$ for the nonincreasing function $a_{\ve}(|x|)$ on $W^+$, which is equal to $\kappa^+$ for $|x|\le1-\ve$ and equal to $1$ for $|x|\ge1+\ve$.
We have \begin{equation} \label{bfirst}
\partial_{x_i} b = b' \frac{x_i}{|x|}, \quad i=1, 2, \qquad \Delta b = b'' + \frac{b'}{|x|} = Ka^{K-1}|\nabla a| + \frac{b'}{|x|}.
\end{equation}
Observe that $|\nabla a|$ is unbounded (i.e. blows up when $\ve \rightarrow 0$) when $|x|$ is close enough to $1$.  
We have the bounds
\begin{equation} \label{bbounds}
 |\nabla b|\le  a^K, \quad 0\le b \le (\kappa^+)^K.
\end{equation} 
Note that $b=0$ on $E$.  Also, since in the end we will take the limit $\ve \rightarrow 0$, we may assume without loss of generality that 
\begin{equation} \label{bb}
0 \le b(|x|) \le 1, \quad \textrm{for } 1-\ve \le |x| \le 1+\ve,
\end{equation}
and in particular this estimate holds whenever $\nabla a\neq 0$.

We now consider the quantity
$$N: =   \frac{ (\alpha T + \beta R)^2}{|x|^2}(1+b) +Au^2,$$
on $W^+$.

On $E$ we have, recalling that $\alpha^2+\beta^2=1$,
$$ \frac{ (\alpha T + \beta R)^2}{|x|^2}(1+b) =  \frac{ (\alpha T + \beta R)^2}{|x|^2} \le |\nabla u|^2,$$
and so  $N \le M$ on $E$.  But from (\ref{atp}) we have $M(p)=N(p)$.  

Assuming that $N$ achieves its maximum value on $W^+$, we may assume it attains this at an interior point.  Indeed, $N$ has at least as large a value at the interior point $p$ than its value on the bottom edge of $W^+$ where $x_2=-1-\delta$, and $N$ is bounded on all other boundary points of $W^+$.
Hence to bound $N$ on $W^+$ it suffices to show that $N$ is uniformly bounded from above at an interior maximum on $W^+$.  It will then follow that $M$ is bounded at $p$ and hence on $E$, giving the required bound for $|\nabla u|$  on $E$.

Suppose then that $N$ achieves an interior maximum at $x\in W^+$.  Clearly we may assume without loss of generality that 
\begin{equation} \label{uA}
|\nabla u(x)| \ge A \ge 1,
\end{equation}
since $A$ will be chosen uniformly.  Observe  that,
\begin{equation} \label{usefulequivalence}
C^{-1}|(\alpha T+ \beta R)(x)| \le |\nabla u(x)| \le C\sqrt{1+b(x)} \, |(\alpha T+ \beta R)(x)|,
\end{equation}
a fact that we will use repeatedly in what follows.  
Indeed, for the upper bound,
$$|\nabla u(x)| \le C\sqrt{M(p)} = C \sqrt{N(p)} \le C \sqrt{N(x)} \le C'\bigg( \sqrt{1+b(x)}\, |(\alpha T + \beta R)(x)| + \sqrt{A}\bigg),$$
where the first inequality used (\ref{firstclaim}) and the last inequality used (\ref{lib}).  Recalling (\ref{uA}), as long as, say
\begin{equation} \label{AC}
A \ge 4(C')^2,
\end{equation}
we have $C'\sqrt{A} \le \frac{1}{2}|\nabla u|(x)$  and 
the second inequality of (\ref{usefulequivalence}) follows.
Here and henceforth, $C$, $C'$, $C_0$  denote uniform constants, which may differ from line to line, which are independent of $\delta, \ve, A$ and  $K$.

We also have at $x$,
\[
\begin{split}
0 = N_i= {} &   \frac{2(\alpha T + \beta R)(\alpha T + \beta R)_i}{|x|^2}(1+b)  + (\alpha T + \beta R)^2 \left( \frac{1+b}{|x|^2} \right)_i + 2Auu_i,
\end{split}
\]
and so
$$(\alpha T + \beta R)_i =   \frac{(\alpha T + \beta R)x_i}{|x|^2} - \frac{(\alpha T +\beta R) b_i}{2(1+b)} - \frac{Au u_i |x|^2}{(1+b)(\alpha T+\beta R)}.$$
Hence we have  the bound at $x$,
\begin{equation} \label{dTR}
| \nabla (\alpha T + \beta R)| \le C a^K | \alpha T + \beta R| + CA,
\end{equation}
where we are using  (\ref{bbounds}) and (\ref{bb}).

 Using  (\ref{bfirst}), we compute at $x$, 
\begin{equation} \label{long}
\begin{split}
0 \ge  \Delta N    =  {} &  \frac{2 (\alpha T + \beta R)(\alpha T + \beta R)_{ii}}{|x|^2} (1+b) +  \frac{2  \left|\nabla (\alpha T + \beta R) \right|^2}{\left| x \right|^2}(1+b)   \\ {} & 
 +  (\alpha T + \beta R)^2  \left( \frac{1}{\left| x \right|^2} \right)_{ii} (1+b)   + \frac{(\alpha T + \beta R)^2}{|x|^2}K a^{K-1} |\nabla a| \\ {} &  + \frac{(\alpha T+ \beta R)^2}{|x|^2}\frac{b'}{|x|}  + 4 (\alpha T+ \beta R) (\alpha T+\beta R)_i \left( \frac{1+b}{\left| x \right|^2} \right)_i  \\ {} & + 2(\alpha T + \beta R)^2 \left( \frac{1}{|x|^2} \right)_i b_i + 2A | \nabla u|^2 - \frac{2Au a_i u_i}{a}\\
=: {} &  \boxed{1} + \boxed{2} + \cdots + \boxed{9}.
\end{split}
\end{equation}
We now choose $$A = C \max\{ (\kappa^+)^{2K},(\kappa^+)^{-2K}, (\kappa^-)^{2K},(\kappa^-)^{-2K}  \},$$ for $C$ sufficiently and uniformly large so that  $\boxed{3} + \boxed{5}+ \boxed{6}+\boxed{7}$ can all be controlled by $\boxed{2}+\boxed{8}$.  (Note that here $A=C(\kappa^+)^{2K}$ suffices but the choice of constant $A$ must  work for all possible cases).  We also assume that $C$ is sufficiently large so that (\ref{AC}) holds.  We will use $\boxed{4}$ to control $\boxed{1}$ and $\boxed{9}$. 
 From Lemma \ref{combination} and (\ref{dTR}), we have
$$\boxed{1} \ge - C_0 |\nabla a| a^K | \alpha T + \beta R|^2  - C A |\nabla a| \, | \alpha T + \beta R|,$$
where we recall that $0\le b\le 1$ when $\nabla a\neq 0$ from (\ref{bb}).

We put this together to obtain
\begin{equation} \label{calcend}
\begin{split}
0 > {} & - C_0 |\nabla a| a^K ( \alpha T + \beta R)^2  - C A |\nabla a| \, | \alpha T + \beta R|  \\ {} & + \frac{(\alpha T + \beta R)^2}{|x|^2} K a^{K-1} |\nabla a| - C A |\nabla a| |\nabla u|.
\end{split}
\end{equation}
Choosing $K=2 \max\{\kappa^+, \kappa^-\} C_0$  gives 
$$|\alpha T + \beta R|(x) \le \frac{C A}{a^K}.$$
Hence $N$ is uniformly bounded from above and this completes the proof in the case when $\kappa^+>1$.  

For the case $0<\kappa^+<1$, replace the definition (\ref{bdefn}) by
\begin{equation*} \label{bdefn2}
b(1+\ve)=0, \quad b'(|x|)= 1-a(|x|)^{-K}\le 0.
\end{equation*}
Then in  (\ref{bfirst}) and (\ref{bbounds}) we now have
 \begin{equation*} \label{bfirst2}
\Delta b= Ka^{-K-1}|\nabla a| + \frac{b'}{|x|}
\end{equation*}
and
\begin{equation*} \label{bbounds2}
 |\nabla b|\le  a^{-K}, \quad 0\le b \le (\kappa^+)^{-K}.
\end{equation*} 
The rest of the proof follows in a similar way.

\section{Remarks on the case $n>2$} \label{section5}

It is an open problem to extend our method to higher dimensions $n>2$.  A missing result is an analogue of Lemma \ref{combination}.  We mention here, without proof, some formulae which  hold for all dimensions $n \ge 2$.   For $j, k=1, \ldots, n$ with $j\neq k$, define
 $$T^{j,k} = -x_k u_j + x_j u_k, \quad  R= \sum_{k=1}^n a x_k u_k ,$$
Then working in a neighborhood of a ball $B^+ \subset \mathbb{R}^n$ centered at the origin, we have the following analogue of Lemma \ref{lemmaTjR}:
$$a(T^{j,k})_{ii}  + a_i (T^{j,k})_i =0 \quad \textrm{and} \quad 
a R_{ii} - a_i R_i = (n-2)aa_iu_i.$$
An analogue of Lemma \ref{lemmaQS2} are the equations
$$-a_k R_j + a_j R_k =  a \sum_{i=1}^n a_i (T^{j,k})_i$$
and 
$$\sum_{j=1}^n a_j R_j + (n-2) \frac{a'R}{|x|} = \frac{a}{2} \sum_{j,k=1}^n \left( a_k (T^{j,k})_j - a_j (T^{j,k})_k \right). 
$$

\end{document}